%% file: main.tex
\documentclass{article}
\input{preamble}

\title{Combinatorial properties of multidimensional continued fractions}
\author{Michele Battagliola, Nadir Murru, Giordano Santilli \\
University of Trento, Department of Mathematics}

\date{}

\begin{document}

\maketitle

\begin{abstract}

The study of combinatorial properties of mathematical objects is a very important research field and continued fractions have been deeply studied in this sense. However, multidimensional continued fractions, which are a generalization arising from an algorithm due to Jacobi, have been poorly investigated in this sense, up to now. In this paper, we propose a combinatorial interpretation of the convergents of multidimensional continued fractions in terms of counting some particular tilings, generalizing some results that hold for classical continued fractions.

\end{abstract}

\input{content}

\input{bibliography}
\end{document}

%% file: preamble.tex
\usepackage[T1]{fontenc}
\usepackage[utf8]{inputenc}
\usepackage{fancyhdr}
\usepackage[toc,page]{appendix}
\usepackage{mathrsfs}
\usepackage[mathscr]{euscript}
\usepackage{graphicx}
\usepackage[dvipsnames]{xcolor}
\usepackage{tikz}
\usepackage{amssymb}
\usepackage{comment}
\usepackage{amsmath}
\usepackage{caption}
\usepackage{subcaption}
\usepackage{amsthm}
\usepackage{enumitem}
\usepackage[english]{babel}
\usepackage{csquotes}
\usepackage{hyperref}
\hypersetup{colorlinks=true}
\usepackage[capitalise]{cleveref}
\usepackage[a4paper,top=3cm,bottom=2cm,left=4cm,right=4cm,marginparwidth=1.75cm]{geometry}
\usepackage{placeins}
\definecolor{lime}{HTML}{A6CE39}
\DeclareRobustCommand{\orcidicon}{%
	\begin{tikzpicture}
	\draw[lime, fill=lime] (0,0) 
	circle [radius=0.16] 
	node[white] {{\fontfamily{qag}\selectfont \tiny ID}};
	\draw[white, fill=white] (-0.0625,0.095) 
	circle [radius=0.007];
	\end{tikzpicture}
	\hspace{-2mm}
}

\foreach \x in {A, ..., Z}{%
	\expandafter\xdef\csname orcid\x\endcsname{\noexpand\href{https://orcid.org/\csname orcidauthor\x\endcsname}{\noexpand\orcidicon}}
}

\usetikzlibrary{arrows, positioning,shapes.misc,calc,shapes.geometric}

 \tikzset{domino/.style={draw=red, rectangle, minimum height=1cm, minimum width=2cm}}
 \tikzset{bar/.style={draw=blue, rectangle, minimum height=1cm, minimum width=3cm}}
 \tikzset{quadrato/.style={draw=ForestGreen, rectangle, minimum height=1cm, minimum width=1cm}}
\renewcommand{\a}{\textbf{a}}
\renewcommand{\b}{\textbf{b}}
\renewcommand{\c}{\textbf{c}}

\theoremstyle{plain} 
\newtheorem{thm}{Theorem}[section]

\newtheorem{cor}[thm]{Corollary} 
\newtheorem{prop}[thm]{Proposition}

\theoremstyle{definition} 
\newtheorem{defn}{Definition}[section]
\newtheorem{oss}{Remark}
\newtheorem{expl}{Example}[section]

%% file: content.tex
\section{Introduction}

Multidimensional continued fractions were introduced by Jacobi \cite{Jac} (and then generalized by Perron \cite{Per}) in the attempt to answer a problem posed by Hermite \cite{Her} who asked for an algorithm that provides periodic representations for algebraic irrationals of any degree, in the same way as continued fractions are periodic if and only if they represent quadratic irrationals. Unfortunately, the Jacobi--Perron algorithm does not solve the problem, which is still a beautiful open problem in number theory, but opened a new and rich research field. Indeed, there are many studies about multidimensional continued fractions and their modifications, aiming to generalize the results and properties of classical continued fractions.

Continued fractions have been widely studied from different points of view.
Several works explore the combinatorial properties of continued fractions giving many interesting interpretations. In the book of Benjamin and Quinn \cite{BQ03}, one chapter is devoted to continued fractions showing that numerators and denominators of convergents count some particular tilings, reporting also some results proved in \cite{BSQ00}. In \cite{Balof}, the author provided further results regarding the properties of continued fractions in terms of counting tilings and giving also a combinatorial interpretation to the expansion of $e$. A different approach to the combinatorial aspects of continued fractions is given in \cite{Flajolet}, where they are connected to some labelled paths. Recently, in \cite{CS20}, the authors described a combinatorial interpretation of continued fractions as quotients of the number of perfect matchings of snake graphs. Further interesting works in this field are \cite{BS22, EK13, F82, Mendez21, PAN12, PS21, Shin10, SZ18}.

Regarding multidimensional continued fractions, there are just few works about their combinatorial properties. In \cite{Ito94}, the Jacobi--Perron algorithm is used for giving a generating method of the so-called stepped surfaces. In \cite{Berthe}, the authors used multidimensional continued fractions for obtaining a method of generation of discrete segments in the three-dimensional space. Finally, in \cite{AB01}, multidimensional continued fractions have been exploited for obtaining results about tilings, discrete approximations of lines and planes,
and Markov partitions for toral automorphism. 

In this paper, we propose an elementary approach to the study of combinatorial properties of multidimensional continued fractions, obtaining a natural interpretation in terms of counting tilings of a board using tiles of length one, two or three, where we can also stack such tiles. We also give an interpretation to negative conditions for the height of the stacks. In particular, Section \ref{sec:pre} is devoted to the preliminary definitions and properties of multidimensional continued fractions, where we also introduce them from a formal point of view. Section \ref{SectionCountingBase} presents the main results.

\section{Preliminaries}
\label{sec:pre}
Multidimensional continued fractions (of degree two) represent a pair of real numbers $(\alpha_0, \beta_0)$ by means of two sequences of integers $(a_i)_{i \geq 0}$, $(b_i)_{i \geq 0}$ as follows:
\[
\alpha_0 = a_0+\cfrac{b_1+\cfrac{1}{a_2+\cfrac{b_3+\cfrac{1}{\ddots}}{a_3+\cfrac{\ddots}{\ddots}}}}{a_1+\cfrac{b_2+\cfrac{1}{a_3+\cfrac{\ddots}{\ddots}}}{a_2+\cfrac{b_3+\cfrac{1}{\ddots}}{a_3+\cfrac{\ddots}{\ddots}}}}, \quad  \beta_0 = b_0 + \cfrac{1}{a_1+\cfrac{b_2+\cfrac{1}{a_3+\cfrac{\ddots}{\ddots}}}{a_2+\cfrac{b_3+\cfrac{1}{\ddots}}{a_3+\cfrac{\ddots}{\ddots}}}}
\]
where the $a_i$'s and $b_i$'s are called \emph{partial quotients} and they can be obtained by the Jacobi algorithm in the following way:
\[
\begin{cases}
\alpha_i = \lfloor a_i \rfloor \cr
\beta_i = \lfloor b_i \rfloor \cr
\alpha_{i+1} = \cfrac{1}{\beta_i - b_i} \cr
\beta_{i+1} = \cfrac{\alpha_i - a_i}{\beta_i - b_i}
\end{cases} 
\quad
i = 0, 1, 2, \ldots
\]
We can introduce multidimensional continued fractions also in a formal way, where the partial quotients are not in general obtained by an algorithm and the numerators are not necessarily equals to 1, as well as in the classical case, given two sequences $(a_i)_{i \geq 0}$, $(b_i)_{i \geq 0}$, one can introduce and study the continued fraction
\[ a_0 + \cfrac{b_1}{a_1 + \cfrac{b_2}{a_2 + \ddots}}. \]

\begin{defn}
Given the sequences of integers $(a_i)_{i \geq 0}$, $(b_i)_{i \geq 0}$ and $(c_i)_{i \geq 0}$ (with $c_0 = 1$), called \emph{partial quotients}, we define the multidimensional continued fraction (MCF) as the following couple of objects:
\begin{equation} \label{eq:MCF}
a_0+\cfrac{b_1+\cfrac{c_2}{a_2+\cfrac{b_3+\cfrac{c_4}{\ddots}}{a_3+\cfrac{\ddots}{\ddots}}}}{a_1+\cfrac{b_2+\cfrac{c_3}{a_3+\cfrac{\ddots}{\ddots}}}{a_2+\cfrac{b_3+\cfrac{c_4}{\ddots}}{a_3+\cfrac{\ddots}{\ddots}}}}, \quad b_0 + \cfrac{c_1}{a_1+\cfrac{b_2+\cfrac{c_3}{a_3+\cfrac{\ddots}{\ddots}}}{a_2+\cfrac{b_3+\cfrac{c_4}{\ddots}}{a_3+\cfrac{\ddots}{\ddots}}}}. 
\end{equation}
In the following, we will write shortly $[(a_0, a_1, \ldots), (b_0, b_1, \ldots), (1, c_1, \ldots)]$ for such a MCF.\\
We also call \emph{complete quotients} the elements of the sequences of real numbers $(\alpha_i)_{i \geq 0}$, $(\beta_i)_{i \geq 0}$ and $(\gamma_i)_{i \geq 0}$ defined by the following relations:
\[ \alpha_i = a_i + \cfrac{\beta_{i+1}}{\alpha_{i+1}}, \quad \beta_i = b_i + \cfrac{\gamma_{i+1}}{\alpha_{i+1}}, \quad \gamma_i = c_i \]
for $i = 0, 1, 2, \ldots$, so that
\((\alpha_0, \beta_0) = [(a_0, a_1, \ldots), (b_0, b_1, \ldots), (1, c_1, \ldots)].\)
\end{defn}

In the following, we will use $\a_i^j$ for denoting the finite sequence $(a_i, a_{i+1}, \ldots, a_j)$, for $i \leq j$ integers. Thus, using this notation the finite MCF $$[(a_0, a_1, \ldots a_n), (b_0, b_1, \ldots, b_n), (1, c_1, \ldots, c_n)]$$ can be also written as $[\a_0^n, \b_0^n, \c_0^n]$

We define the $n$--th convergent of a MCF, similarly to the convergents of classical continued fractions, as the following pair of rationals:\\
\[ \left( \frac{A(\a_0^n, \b_0^n, \c_0^n)}{C(\a_0^n, \b_0^n, \c_0^n)}, \frac{B(\a_0^n, \b_0^n, \c_0^n)}{C(\a_0^n, \b_0^n, \c_0^n)} \right) := [\a_0^n, \b_0^n, \c_0^n]. \]
Sometimes, when there is no possibility of confusion, we will also use the notation $\left( \cfrac{A_n}{C_n}, \cfrac{B_n}{C_n} \right)$ without making explicit the dependence on the partial quotients.

\begin{oss} \label{rem:mat}
We would like to observe that the partial quotient $c_0$ does not appear in the expansion of the MCF described in \eqref{eq:MCF}. However, we set it equals to 1 because the convergents can be also evaluated in the following way:
\[ \begin{pmatrix} a_0 & 1 & 0 \cr b_0 & 0 & 1 \cr c_0 & 0 & 0 \end{pmatrix} \cdots \begin{pmatrix} a_n & 1 & 0 \cr b_n & 0 & 1 \cr c_n & 0 & 0 \end{pmatrix} = \begin{pmatrix} A_n & A_{n-1} & A_{n-2} \cr B_n & B_{n-1} & B_{n-2} \cr C_n & C_{n-1} & C_{n-2} \end{pmatrix}.\]
Since $\frac{A_0}{C_0} = a_0$ and $\frac{B_0}{C_0} = b_0$, it is a natural choice to set $c_0 = 1$. Moreover, we would like to highlight that $A_n$ does not depend on $b_0, c_0, c_1$ and $B_n$ does not depend on $a_0, b_1, c_0, c_2$.
\end{oss}

\begin{prop} \label{prop:conv}
Given the sequences of integers $(a_i)_{i \geq 0}$, $(b_i)_{i \geq 0}$, $(c_i)_{i \geq 0}$, then for all $n \geq 3$ we have
\begin{equation} \label{eq:conv1}
A(\a_0^n, \b_0^n, \c_0^n) = a_0 A(\a_1^n, \b_1^n, \c_1^n)+ b_1 A(\a_2^n, \b_2^n, \c_2^n) + c_2 A(\a_3^n, \b_3^n, \c_3^n) \end{equation}
and
\small
\begin{equation} \label{eq:conv2}
A(\a_0^n, \b_0^n, \c_0^n) = a_n A(\a_0^{n-1}, \b_0^{n-1}, \c_0^{n-1}) + b_n A(\a_0^{n-2}, \b_0^{n-2}, \c_0^{n-2}) + c_n A(\a_0^{n-3}, \b_0^{n-3}, \c_0^{n-3}).
\end{equation}
\normalsize
\end{prop}
\begin{proof}

By definition we have that
\begin{equation*}
        \frac{A(\a_0^n, \b_0^n, \c_0^n)}{C(\a_0^n, \b_0^n, \c_0^n)}=a_0+\frac{B(\a_1^n, \b_1^n, \c_1^n)}{A(\a_1^n, \b_1^n, \c_1^n)}, \quad \frac{B(\a_0^n, \b_0^n, \c_0^n)}{C(\a_0^n, \b_0^n, \c_0^n)}=b_0+c_1\frac{C(\a_1^n, \b_1^n, \c_1^n)}{A(\a_1^n, \b_1^n, \c_1^n)}.
\end{equation*}
Thus, we have the following equalities:
    \begin{flalign} \label{RicorsioneBase}
        C(\a_0^n, \b_0^n, \c_0^n) &= A(\a_1^n, \b_1^n, \c_1^n)\nonumber\\
        B(\a_0^n, \b_0^n, \c_0^n) &= c_1C(\a_1^n, \b_1^n, \c_1^n) +b_0A(\a_1^n, \b_1^n, \c_1^n)\\
        A(\a_0^n, \b_0^n, \c_0^n) &= a_0A(\a_1^n, \b_1^n, \c_1^n) +B(\a_1^n, \b_1^n, \c_1^n).\nonumber
    \end{flalign}
By substitution we get
    \begin{align*}\label{EquazioneContoTassellamentiBase}
        A(\a_0^n, \b_0^n, \c_0^n) &= a_0A(\a_1^n, \b_1^n, \c_1^n)+b_1A(\a_2^n, \b_2^n, \c_2^n)+c_2A(\a_3^n, \b_3^n, \c_3^n).
    \end{align*}
Equation \eqref{eq:conv2} can be proved by induction. The basis of the induction is trivial. By inductive hypothesis we have that
    \begin{align*}
     A(\a_1^n, \b_1^n, \c_1^n) &= a_nA(\a_1^{n-1}, \b_1^{n-1}, \c_1^{n-1})+b_nA(\a_1^{n-2}, \b_1^{n-2}, \c_1^{n-2})+c_nA(\a_1^{n-3}, \b_1^{n-3}, \c_1^{n-3})
    \end{align*}
      \begin{align*}
     A(\a_2^n, \b_2^n, \c_2^n) &= a_nA(\a_2^{n-1}, \b_2^{n-1}, \c_2^{n-1})+b_nA(\a_2^{n-2}, \b_2^{n-2}, \c_2^{n-2})+c_nA(\a_2^{n-3}, \b_2^{n-3}, \c_2^{n-3})
    \end{align*}
    \begin{align*}
     A(\a_3^n, \b_3^n, \c_3^n) &= a_nA(\a_3^{n-1}, \b_3^{n-1}, \c_3^{n-1})+b_nA(\a_3^{n-2}, \b_3^{n-2}, \c_3^{n-2})+c_nA(\a_3^{n-3}, \b_3^{n-3}, \c_3^{n-3})
    \end{align*}
    Substituting and factoring out $a_n,b_n$ and $c_n$ we get
      \begin{align*}
      &A(\a_0^n, \b_0^n, \c_0^n) = a_n \cdot (a_0A(\a_1^{n-1}, \b_1^{n-1}, \c_1^{n-1})+b_1A(\a_2^{n-1}, \b_2^{n-1}, \c_2^{n-1})+c_2A(\a_3^{n-1}, \b_3^{n-1}, \c_3^{n-1})+\\
      &+b_n \cdot (a_0A(\a_1^{n-2}, \b_1^{n-2}, \c_1^{n-2})+b_1A(\a_2^{n-2}, \b_2^{n-2}, \c_2^{n-2})+c_2A(\a_3^{n-2}, \b_3^{n-2}, \c_3^{n-2}))+\\
      &+c_n \cdot (a_0A(\a_1^{n-3}, \b_1^{n-3}, \c_1^{n-3})+b_1A(\a_2^{n-3}, \b_2^{n-3}, \c_2^{n-3})+c_2A(\a_3^{n-3}, \b_3^{n-3}, \c_3^{n-3})).
    \end{align*}
    Finally, by using again \eqref{eq:conv1} we get the thesis. 
\end{proof}

\section{Counting the number of tilings using multidimensional continued fractions}\label{SectionCountingBase}

In this section, we give a combinatorial interpretation to the convergents of a MCF in terms of tilings of some boards, extending the approaches of Benjamin \cite{BQ03, BSQ00} and Balof \cite{Balof} for the classical continued fractions. 

In the following, a $(n+1)$--board is a $1 \times (n+1)$ chessboard, a \textit{square} is a $1 \times 1$ tile, a \textit{domino} is a $1 \times 2$ tile and a \textit{bar} is a $1 \times 3$ tile. The $n + 1$ cells of the $(n+1)$--board are labeled from 0 to $n$ (i.e., we refer to the cell of position 0 for the first cell and so on).
A tiling of a $n$--board is a covering using squares, dominoes and bars that can be also stacked. In particular, the height conditions for stacking them are given by finite sequences like $(\a_0^{n}, \b_0^{n}, \c_0^{n})$, where
\begin{itemize}
    \item the element $a_i$ of $\a_0^{n}$ denotes the number of stackable squares in the $i$--th position (e.g., $a_0$ is the number of stackable squares in the cell of position 0 of the $(n+1)$--board);
    \item the element $b_i$ of $\b_0^{n}$ denotes the number of stackable dominoes covering the positions $i-1$ and $i$ (e.g., $b_1$ is the number of stackable dominoes covering the positions 0 and 1 of the $(n+1)$--board);
    \item the element $c_i$ of $\c_0^{n}$ denotes the number of stackable bars covering the positions $i-2$, $i-1$ and $i$ (e.g., $c_2$ is the number of stackable dominoes covering the positions 0, 1, and 2 of the $(n+1)$--board).
\end{itemize}
At first glance the first element in $\b_0^{n}$ does not give height conditions, as well as the first two elements of $\c_0^{n}$, however their role will be important later when discussing different types of tilings.
We will denote by $M(\a_0^{n}, \b_0^{n}, \c_0^{n})$ the number of possible tilings of a $(n+1)$--board with height conditions $(\a_0^{n}, \b_0^{n}, \c_0^{n})$.

\input{disegno}

\begin{thm}\label{TeoFrazioniTiling}
Let  $\left( \frac{A(\a_0^n, \b_0^n, \c_0^n)}{C(\a_0^n, \b_0^n, \c_0^n)}, \frac{B(\a_0^n, \b_0^n, \c_0^n)}{C(\a_0^n, \b_0^n, \c_0^n)} \right) := [\a_0^n, \b_0^n, \c_0^n]$. Then we have the following:
\begin{itemize}
    \item $A(\a_0^n, \b_0^n, \c_0^n)$ counts the number of possible tilings of a $(n+1)$--board with height conditions $(\a_0^n, \b_0^n, \c_0^n)$. 
    \item $B(\a_0^n, \b_0^n, \c_0^n)$  counts the number of possible tilings of a $(n+2)$--board, where only in this case the first cell is labelled with -1 (i.e., we add a cell on the left to a $(n+1)$--board), with height conditions $(\a_0^{n}, \b_0^{n}, \c_0^{n})$, such that the first tile of the tiling is a domino or a bar. 
    \item $C(\a_0^n, \b_0^n, \c_0^n)$ counts the number of possible tilings of a $n$--board with height conditions $(\a_1^n, \b_1^n, \c_1^n)$.
 
\end{itemize}
\end{thm}
\begin{proof}
    We want to show that the number of tilings $M(\a_0^n, \b_0^n, \c_0^n)$ and $A(\a_0^n, \b_0^n, \c_0^n)$ have the same initial values and recurrence formula. 
    Clearly, for a $1$--board we have
    $$M(a_0, b_0, c_0)= a_0 = A(a_0, b_0, c_0),$$ 
    and for a $2$--board
    $$M(\a_0^1,\b_0^1,\c_0^1) = a_0 a_1 + b_1 = A(\a_0^1,\b_0^1,\c_0^1).$$ 
    Then for a $3$--board we can have tilings with $3$ stacks of squares, or $1$ stack of squares in the first position and $1$ stack of dominoes in the second and third position, or one stack of dominoes in the first and second position and one stack of squares in the third position, or $1$ stack of bars:
    $$M(\a_0^2,\b_0^2,\c_0^2) = a_0a_1a_2 + a_0b_2 + a_2b_1 + c_2 = A(\a_0^2,\b_0^2,\c_0^2).$$
    
    For a $(n+1)$--board, with $n > 2$, we can observe that the number of tilings satisfies the following recursive formula:
    $$ M(\a_0^n, \b_0^n, \c_0^n) = a_0M(\a_1^n, \b_1^n, \c_1^n)+b_1M(\a_2^n, \b_2^n, \c_2^n)+c_1M(\a_3^n, \b_3^n, \c_3^n),$$
    since we can count the tilings dividing them in three sets: tilings that start with a stack of squares, tilings that start with a stack of dominoes and tilings that start with a stack of bars. Thus, the number of tilings of a $(n+1)$--board starting with a stack of squares is $a_0 M(\a_1^n, \b_1^n, \c_1^n)$ and similarly for the other two situations. So we have the first point. The third point follows immediately from the first equality in \eqref{RicorsioneBase}.
    
    About the second point we can observe that if the board has only one cell (i.e. the -1 cell) there are no possible tilings ($a_{-1} $ is implicitly set at 0), and this is consistent with $B_{-1} = 0$ (see the matricial representation of the convergents in \cref{rem:mat}. Moreover, $M(\a_{-1}^0, \b_{-1}^0, \c_{-1}^0) = b_0 = B(a_0, b_0, c_0)$, since we can tile the $2$--board only with a domino. Similarly, $M(\a_{-1}^1, \b_{-1}^1, \c_{-1}^1) = b_0a_1 + c_1 = B(\a_0^1, \b_0^1, \c_0^1)$, because we only have two possibilities: a tile composed by one domino and one square or a tile composed by one bar. Now, we can complete the proof by induction with the same argument used above.

\end{proof}

With the same notation as \cref{TeoFrazioniTiling} we have the following corollary for a $(n+1)$--circular board, which is a $(n+1)$--board where the first and last tile are bordering.

\begin{cor}
 The number of tilings of a $n+1$--circular board with height condition $(\a_0^n, \b_0^n, \c_0^n)$ with $c_0=0$ (i.e. we forbid bars covering the cells $0,n,n-1$) is $A(\a_0^n, \b_0^n, \c_0^n)+B(\a_0^{n-1}, \b_0^{n-1}, \c_0^{n-1})$
\end{cor}

\begin{proof}
    $A(\a_0^n, \b_0^n, \c_0^n)$ counts the number of all the possible tilings where the cells $0$ and $n$ are not covered by the same stack of dominoes or bars. The tilings that are missing are the ones where a stack of dominoes covers $0$ and $n$ or a stack of bars covers $1,0,n$, (the only other possible case, where a stack of bars covers $0,n,n-1$) is impossible since $c_0=0$). In particular we notice that in both case the stack begins in the cell $n$.
    By the previous theorem $B(\a_0^{n-1}, \b_0^{n-1}, \c_0^{n-1})$ counts the number of tilings of a $n+1$ board starting from cell $-1$, starting with a stack of dominoes or bars. We can notice that this is the same of saying that the board starts with cell $n$ followed by the cell $0$.
\end{proof}

In the following proposition we show that the numerators of convergents of a MCF can be also seen in terms of permutations.

\begin{prop}
    If $a_0 = 4, b_1 = 1, c_2 = 1$ and $a_i=i+1$ for $i>0$, $b_i=i-1$ for $i>1$, and $c_i=i-2$ for $i>2$, then $A_n=(n+2)!+(n+1)!+n!$, i.e. $A_0=4, A_1=9, A_2=32,...$.
\end{prop}
\begin{proof}
    We prove the identity by induction. It is straightforward to check the thesis for $A_0, A_1, A_2$.
    We will now suppose $A_k = (k+2)!+(k+1)!+k!$ for every $k<n$ and prove the property for $n$.
    By $\eqref{eq:conv2}$ we have
    $$A_n = a_n A_{n-1} + b_n A_{n-2} + c_n A_{n-3}.$$
    From the definition of  $a_i$, $b_i$, $c_i$ and the inductive hypothesis we get
    \begin{multline*}
         A_n = (n+1)[(n+1)!+n!+(n-1)!] + (n-1)[n!+(n-1)!+(n-2)!]+\\ +(n-2)[(n-1)!+(n-2)!+(n-3)!].
    \end{multline*}
    We will deal with the three addends separately:
    \begin{multline*}
        (n+1)[(n+1)!+n!+(n-1)!] = (n+1)(n+1)!+(n+1)!+n!+(n-1)!\\
        =(n+2)!-(n+1)!+(n+1)!+n!+(n-1)!=(n+2)!+n!+(n-1)!,
    \end{multline*}
    $$(n-1)[n!+(n-1)!+(n-2)!] = (n-1)(n-2)!(n(n-1)+(n-1)+1)=n!n,$$
    $$(n-2)[(n-1)!+(n-2)!+(n-3)!] = (n-2)![(n-1)(n-2)+(n-2)+1] = (n-1)!(n-1).$$
    Summing all three equation we get
    $$A_n = (n+2)!+n!+(n-1)!+n!n+(n-1)!(n-1) = (n+2)!+(n+1)!+n!$$
\end{proof} 

\begin{oss}
The MCF of the previous proposition is
\[ [(4, 2, 3, 4, 5, 6, \ldots), (b_0, 1, 1, 2, 3, 4, \ldots), (1, c_1, 1, 1, 2, 3, \ldots)] \]
and the first sequence of convergents $\left(\frac{A_n}{C_n}\right)_{n \geq 0}$ appears to be convergent to the real number $4.54752\ldots$, but we were not able to explicitly determine this real number. In the case of classical continued fraction a similar situation happens for the continued fraction
\[ 2 + \cfrac{1}{1 + \cfrac{1}{2 + \cfrac{2}{3 + \cfrac{3}{4 + \ddots}}}} \]
whose convergents have as numerator the sequence $((n + 1)! + n!)_{n \geq 0}$ and in this case the continued fraction converges to $e$.
\end{oss}

\subsection{Negative Dominoes and Bars}
Now, we want to generalize \Cref{TeoFrazioniTiling} in order to allow negative $b_i, c_i$, following the ideas of \cite{EustisThesis}.

We notice that a positive $b_i$ adds $b_i$ number of ways to tile cells $i-1,i$. So a natural way to explain negative coefficient is to impose some restrictions such that a negative $b_i$ give us $|b_i|$ less way to cover the cells $i-1,i$. An analogous argument can be done for $c_i$.

\begin{defn}[Mixed Tiling]\label{DefMixedTiling}
Let  $(a_i)_{i \geq 0}$ be a sequence of positive integers and $(b_i)_{i \geq 0}$, $(c_i)_{i \geq 0}$ be sequences of integers such that
\begin{itemize}
    \item if $b_i<0$ and $c_i>0$, then $a_i>|b_i|$;
    \item if $b_i>0$ and $c_i<0$, then either $a_i>|c_i|$ or $b_i>|c_i|$;
    \item if $b_i<0$ and $c_i<0$, then $a_i>|b_i|+|c_i|$.
\end{itemize}
Then we define a \emph{mixed tiling} of an $(n + 1)-$board with height condition respectively given by $\a_0^n$, $\b_0^n$ and $\c_0^n$ as follows: for any $k \in \mathbb{N}$,
\begin{enumerate}
    \item if $b_k \ge 0$ and $c_k \ge 0$, we fall back in the same case defined at the beginning of \Cref{SectionCountingBase};
    \item if $b_k < 0$ and $c_k>0$, when there is a stack of $a_{k-1}$ squares in the cell $k-1$, we discard the tilings having up to $|b_k|$ squares in the cell $k$ and we refer to them as inadmissible tilings;
    \item if $c_k < 0$ and $b_k>0$, we have two cases:
     \begin{enumerate}
        \item if $a_k>|c_k|$, when there is a stack stack of $a_{k-2}$ squares in the cell $k-2$ and a stack of $a_{k-1}$ squares in the cell $k-1$, then we consider as inadmissible all the tilings having up to $|c_k|$ squares in the cell $k$;
        \item otherwise, necessarily $b_k>|c_k|$. In this case when there is a stack of $a_{k-2}$ squares in the cell $k-2$, the inadmissible tilings are those with up to $|c_k|$ dominoes covering the cells $k-1,k$;
    \end{enumerate}
    \item if $c_k < 0$ and $b_k<0$, we have two cases:
    \begin{enumerate}
    \item when at the same time there is a stack of $a_{k-2}$ squares in the cell $k-2$ and a stack of $a_{k-1}$ squares in the cell $k-1$, we discard all the tilings having up to $|c_k|+|b_k|$ squares in the cell $k$;
    \item when there is a stack of $a_{k-1}$ squares in the cells $k-1$, the inadmissible tilings have up to $|b_k|$ squares in the cell $k$.
    \end{enumerate}
\end{enumerate}
\end{defn}
\begin{oss}
Notice that the last condition applies when there are less than $a_{k-2}$ squares in the cell $k-2$ to compensate the negative $b_k$ as in the case 4a. 
\end{oss}

\input{disegno2.tex}

\begin{thm}\label{TeoMixedTiling}
Consider the height conditions given by $\left( \frac{A(\a_0^n, \b_0^n, \c_0^n)}{C(\a_0^n, \b_0^n, \c_0^n)}, \frac{B(\a_0^n, \b_0^n, \c_0^n)}{C(\a_0^n, \b_0^n, \c_0^n)} \right) := [\a_0^n, \b_0^n, \c_0^n]$ such that the conditions in \cref{DefMixedTiling} hold. 
Then $A(\a_0^n, \b_0^n, \c_0^n)$ is the number of mixed tilings with height conditions $\a_0^n$, $\b_0^n$, and $\c_0^n$.
\end{thm}

\begin{proof}
    In the following proof we will exclude the case of $b_n$ and $c_n$ being both not negative, since it follows easily by \cref{TeoFrazioniTiling}. 

    First we want to show that the number of mixed tiling $M(\a_0^n, \b_0^n, \c_0^n)$ satisfies the same initial condition and recurrence relations of $A(\a_0^n, \b_0^n, \c_0^n)$. 
    \begin{itemize}
        \item If $n=0$ we trivially have $M(a_0, b_0, c_0) = a_0 =A(a_0, b_0, c_0)$.
        \item If $n=1$ we have $M(\a_0^1, \b_0^1, \c_0^1)= a_0a_1+b_1 = A(\a_0^1, \b_0^1, \c_0^1)$, since $b_1 <0$ we may cover using only squares, that are $a_0a_1$, but we need to subtract $|b_1|$ inadmissible tilings, when we have $a_0$ squares in the cell in position 0 and less than $|b_1|+1$ squares in the cell in position 1.
      \item If $n=2$ we have $M(\a_0^2, \b_0^2, \c_0^2) = a_0a_1a_2+a_0b_2+a_2b_1+c_2 = A(\a_0^2, \b_0^2, \c_0^2)$, indeed $a_0a_1a_2$ is the total number of tilings consisting in only squares, $a_0b_2$ and $b_1a_2$ are the number of tilings involving a stack of squares and a stack of dominoes that we need to add (when $b_i\ge 0$) or subtract (when $b_i<0$). Finally $c_2$ is the number of tiling using only bars we need to add (when $c_2 \ge 0$)) or the number of tilings we need to subtract ($c_2<0$).
    \end{itemize}
    
    We now need to prove that $M$ has the same recurrence property expressed in Proposition \ref{prop:conv}.
 
    \begin{itemize}
    \item If $b_n<0$ and $c_n>0$, then every tiling must finish either with a stack of squares or a stack of bars. 
    By induction there are $c_nM(\a_0^{n-3}, \b_0^{n-3}, \c_0^{n-3})$ tilings that end with a stack of bars and $a_nM(\a_0^{n-1}, \b_0^{n-1}, \c_0^{n-1})$ tilings that end with a stack of squares, ignoring the condition stated in \cref{DefMixedTiling}. Among these, we need to subtract $|b_n|M(\a_0^{n-2}, \b_0^{n-2}, \c_0^{n-2})$ inadmissible tiling, namely those having a stack of $a_{n-1}$ squares in the cell $n-1$ and less than $|b_n|+1$ squares in the cell $n$.
    
    \item If $b_n>0$ and $c_n<0$ then every tiling must finish either with a stack of squares or a stack of dominoes. By induction these are respectively  $a_nM(\a_0^{n-1}, \b_0^{n-1}, \c_0^{n-1})$ and $b_nM(\a_0^{n-2}, \b_0^{n-2}, \c_0^{n-2})$. Now we need to distinguish two possible cases:
    \begin{itemize}
   \item If $a_n>|c_n|$, then we need to subtract $|c_n|M(\a_0^{n-3}, \b_0^{n-3}, \c_0^{n-3})$ inadmissible tilings, i.e. those having a stack of $a_{n-1}$ squares in  cell $n-1$, $a_{n-2}$ squares in the cell $n-2$ and less than $|c_n|+1$ squares in  cell $n$. 
   \item If $a_n\le|c_n|$, then $b_n>|c_n|$ by hypothesis and so we need to subtract  $|c_n|M(\a_0^{n-3}, \b_0^{n-3}, \c_0^{n-3})$ inadmissible tiling, which in this case are those having stack of $a_{n-2}$ squares in the cell $n-2$ and less than $|c_n|+1$ dominoes covering the cells in positions $n-1$, $n$.
   \end{itemize}
   
    \item Finally, if $c_n<0$ and $b_n<0$ then every tiling must finish with a stack of squares. By induction there are $a_nM(\a_0^{n-1}, \b_0^{n-1}, \c_0^{n-1})$ tilings that end with a stack of squares.  From this number we need to subtract $|b_n|M(\a_0^{n-2}, \b_0^{n-2}, \c_0^{n-2})$ inadmissible tilings, those when there is a stack of $a_{n-1}$ squares in the cell $n-1$ and less than $|b_n|+1$ squares in the cell $n$. Moreover we also need to subtract $(|b_n|+|c_n|)M(\a_0^{n-3}, \b_0^{n-3}, \c_0^{n-3})$ inadmissible tilings, i.e. when there is are stacks of $a_{n-1}$ and $a_{n-2}$ squares in the cells $n-1$ and $n-2$ respectively and less than $|c_n|+|b_n|+1$ squares in the cell $n$. 
    However in this way we have counted twice the tilings having full stacks of $a_{n-1}$ and $a_{n-2}$ squares in the cells $n-1$ and $n-2$, and less than $|b_n|+1$ squares in the last cell, so we have to add up this coverings again. These are a total of $|b_n|M(\a_0^{n-3}, \b_0^{n-3}, \c_0^{n-3})$ tiling, obtaining the result stated by the thesis.
\end{itemize}

\end{proof}

\begin{oss}
The Jacobi algorithm has been generalized to higher dimensions by Perron \cite{Per} as follows:
\[
\begin{cases}
a_n^{(i)} = \lfloor \alpha_n^{(i)} \rfloor \cr
\alpha_{n+1}^{(1)} = \cfrac{1}{\alpha_n^{(m)} - a_n^{(m)}} \cr
\alpha_{n+1}^{(i)} = \cfrac{\alpha_n^{(i-1)} - a_n^{(i-1)}}{\alpha_n^{(m)} - a_n^{(m)}}
\end{cases}
\quad n = 0, 1, 2, \ldots
\]
starting from $m$ real numbers $\alpha_0^{(1)}, \ldots, \alpha_0^{(m)}$ and providing a MCF $$[(a_0^{(1)}, a_1^{(1)}, \ldots), \ldots, (a_0^{(m)}, a_1^{(m)}, \ldots)]$$ of degree $m$ which is defined by the following relation
\[
\begin{cases}
\alpha_n^{(i-1)} = a_n^{(i-1)} + \cfrac{\alpha_{n+1}^{i}}{\alpha_{n+1}^{(1)}}, \quad i = 2, \ldots, m \cr
\alpha_n^{(m)} = a_n^{(m)} + \cfrac{1}{\alpha_{n+1}^{(1)}}
\end{cases}
\quad n = 0, 1, 2, \ldots
\]
Our results about the MCF of degree 2 easily extends to a MCF of degree $m$ by considering $m + 1$ different tiles of length $1, 2, \ldots, m + 1$. In this paper we deal with the case of degree 2 for the seek of simplicity about the notation. 
\end{oss}

%% file: disegno.tex
\begin{expl}
Consider $n=5$ with the following height conditions $$\a_0^5=(1,2,3,2,3,2), \quad \b_0^5=(-,6,5,4,3,2), \quad \c_0^5=(-,-,1,2,3,1),$$ 
where we do not explicit the values of $b_0, c_0, c_1$ since, in this case, they are not relevant for the tilings.
Examples of valid tilings are represented in \cref{fig:ex2}, while in \cref{fig:ex5} is represented a non-valid tiling for these height conditions: in this case there are too many bars in the last tiles.
\FloatBarrier
\begin{comment}
\begin{figure}[!hbt]
    \centering
   \begin{tikzpicture}
\node[domino, label={[xshift=-0.5cm, yshift=-1.5cm]$0$}, label={[xshift=0.5cm, yshift=-1.5cm]$1$}] (a) {} ;
\node[domino, above= 0cm of a] (a2) {} ;
\node[domino,above= 0cm of a2] (a3) {} ;
\node[quadrato, right =0cm of a,label={[yshift=-1.5cm]$2$}] (b) {} ;
\node[bar, right =0cm of b,label={[xshift=-1cm, yshift=-1.5cm]$3$},label={[yshift=-1.5cm]$4$},label={[xshift=1cm, yshift=-1.5cm]$5$}] (c) {} ;
\node[bar, above =0cm of c] (c2) {} ;

\end{tikzpicture}
    \caption{Example of a possible tiling.}
    \label{fig:ex1}
\end{figure}
\end{comment}

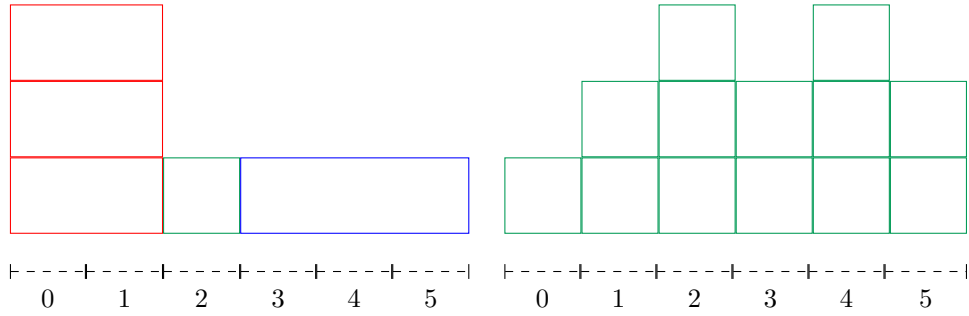
\begin{figure}[!hbt]
    \centering
\begin{subfigure}{.5\textwidth}
   \begin{tikzpicture}
\node[domino] (a) {} ;
\node[domino, above= 0cm of a] (a2) {} ;
\node[domino,above= 0cm of a2] (a3) {} ;
\node[quadrato, right =0cm of a] (b) {} ;
\node[bar, right =0cm of b] (c) {} ;
%\node[bar, above =0cm of c] (c2) {} ;

\node [below=0.5cm of a.south west] (inizio) {}; 
\node [right=1cm of inizio.west] (n1) {}; 
\node [below=0.5cm of a.south east] (n2) {}; 
\node [below=0.5cm of b.south east] (n3) {}; 
\node [right=1cm of n3.west] (n4) {}; 
\node [right=1 cm of n4.west] (n5) {}; 
\node [below=0.5cm of c.south east] (fine) {}; 
\draw[|-|,dashed] (inizio.north) -- node[label=below:$0$]{} (n1.north);
\draw[|-|,dashed] (n1.north)-- node[label=below:$1$]{} (n2.north); 
\draw[|-|,dashed] (n2.north)-- node[label=below:$2$]{}(n3.north);
\draw[|-|,dashed] (n3.north)-- node[label=below:$3$]{}(n4.north);
\draw[|-|,dashed] (n4.north)-- node[label=below:$4$]{}(n5.north); 
\draw[|-|,dashed] (n5.north)-- node[label=below:$5$]{}(fine.north); 

\end{tikzpicture}
\end{subfigure}%
\hfill
\begin{subfigure}{.5\textwidth}
  \begin{tikzpicture}
\node[quadrato] (a) {} ;
\node[quadrato, right =0cm of a] (b) {};
\node[quadrato, above =0cm of b] (b2) {} ;
\node[quadrato, right =0cm of b] (c) {} ;
\node[quadrato, above =0cm of c] (c2) {} ;
\node[quadrato, above =0cm of c2] (c3) {} ;
\node[quadrato, right =0cm of c] (d) {} ;
\node[quadrato, above =0cm of d] (d2) {} ;
%\node[quadrato, above =0cm of d2] (d3) {} ;
%\node[quadrato, above =0cm of d3] (d4) {} ;
\node[quadrato, right =0cm of d] (e) {} ;
\node[quadrato, above =0cm of e] (e2) {} ;
\node[quadrato, above =0cm of e2] (e3) {} ;
%\node[quadrato, above =0cm of e3] (e4) {} ;
%\node[quadrato, above =0cm of e4] (e5) {} ;
\node[quadrato, right =0cm of e] (f) {} ;
\node[quadrato, above =0cm of f] (f2) {} ;
%\node[quadrato, above =0cm of f2] (f3) {} ;
%\node[quadrato, above =0cm of f3] (f4) {} ;
%\node[quadrato, above =0cm of f4] (f5) {} ;
%\node[quadrato, above =0cm of f5] (f6) {} ;

\node [below=0.5cm of a.south west] (inizio) {}; 
\node [right=1cm of inizio.west] (n1) {}; 
\node [right=1cm of n1.west] (n2) {}; 
\node [right=1cm of n2.west] (n3) {}; 
\node [right=1cm of n3.west] (n4) {}; 
\node [right=1 cm of n4.west] (n5) {}; 
\node [below=0.5cm of f.south east] (fine) {}; 
\draw[|-|,dashed] (inizio.north) -- node[label=below:$0$]{} (n1.north);
\draw[|-|,dashed] (n1.north)-- node[label=below:$1$]{} (n2.north); 
\draw[|-|,dashed] (n2.north)-- node[label=below:$2$]{}(n3.north);
\draw[|-|,dashed] (n3.north)-- node[label=below:$3$]{}(n4.north);
\draw[|-|,dashed] (n4.north)-- node[label=below:$4$]{}(n5.north); 
\draw[|-|,dashed] (n5.north)-- node[label=below:$5$]{}(fine.north); 

\end{tikzpicture}
\end{subfigure}%
    \caption{Examples of possible tilings.}
    \label{fig:ex2}
\end{figure}

\begin{figure}[!hbt]
    \centering
   \begin{tikzpicture}
\node[domino] (a) {} ;
\node[domino, above= 0cm of a] (a2) {} ;
%\node[domino,above= 0cm of a2] (a3) {} ;
\node[quadrato, right =0cm of a] (b) {} ;
\node[quadrato, above =0cm of b] (b2) {} ;
\node[quadrato, above =0cm of b2] (b3) {} ;
\node[bar, right =0cm of b] (c) {} ;
\node[bar, above =0cm of c] (c2) {} ;

\node [below=0.5cm of a.south west] (inizio) {}; 
\node [right=1cm of inizio.west] (n1) {}; 
\node [below=0.5cm of a.south east] (n2) {}; 
\node [below=0.5cm of b.south east] (n3) {}; 
\node [right=1cm of n3.west] (n4) {}; 
\node [right=1 cm of n4.west] (n5) {}; 
\node [below=0.5cm of c.south east] (fine) {}; 
\draw[|-|,dashed] (inizio.north) -- node[label=below:$0$]{} (n1.north);
\draw[|-|,dashed] (n1.north)-- node[label=below:$1$]{} (n2.north); 
\draw[|-|,dashed] (n2.north)-- node[label=below:$2$]{}(n3.north);
\draw[|-|,dashed] (n3.north)-- node[label=below:$3$]{}(n4.north);
\draw[|-|,dashed] (n4.north)-- node[label=below:$4$]{}(n5.north); 
\draw[|-|,dashed] (n5.north)-- node[label=below:$5$]{}(fine.north); 

\end{tikzpicture}
    \caption{Example of a non-admissible tiling.}
    \label{fig:ex5}
\end{figure}
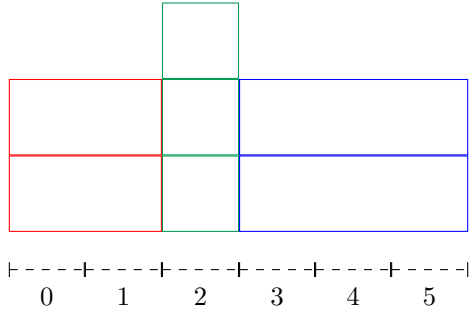

\FloatBarrier

\end{expl}

%% file: disegno2.tex
\begin{expl}
Consider the height conditions given by the following continued fraction: 
\[
(\a=[2,3,1,2,2,3],\b=[-,-1,3,3,2,-1],\c=[-,-,-2,2,1,-1]).
\]
In this case there are several restrictions given by these choice of conditions: 
\begin{itemize}
    \item Since $b_1=-1$, then when we have $a_0=2$ squares in position $0$, we need to exclude the case having $1$ square in the tile in position $1$ (see \cref{fig:n7}).
    \item Since $c_2=-2$ and $a_2=1<|c_2|$, then we are in the case 3b and we need to exclude the tilings having $2$ squares in position $0$ and $1$ or $2$ dominoes in the positions $1$ and $2$ (see \cref{fig:n8}).
    \item Finally, $b_5=c_5=-1$ so we are in the fourth case. Therefore the inadmissible tilings are those having $a_3=2$ squares in position $3$, $a_4=2$ squares in position $4$ and $1$ or $2$ squares in position $5$ (see \cref{fig:n9}). Moreover we also need to discard the tilings having $a_4=2$ squares in position $4$ and $|b_5|=1$ square in position $5$ (see \cref{fig:n10}).
\end{itemize}

\end{expl}

\begin{figure}[!hbt]
    \centering
\begin{subfigure}{.5\textwidth}
   \begin{tikzpicture}
\node[quadrato] (a) {} ;
\node[quadrato, above= 0cm of a] (a2) {} ;
\node[quadrato, fill=red, right =0cm of a] (b) {} ;
\node[quadrato, right =0cm of b] (c) {} ;
\node[domino,right=0cm of c] (d) {};
\node[domino,above=0cm of d] (d2) {};
\node[quadrato, right =0cm of d] (e) {} ;

\node [below=0.5cm of a.south west] (inizio) {}; 
\node [right=1cm of inizio.west] (n1) {}; 
\node [right=1cm of n1.west] (n2) {}; 
\node [below=0.5cm of c.south east] (n3) {}; 
\node [right=1cm of n3.west] (n4) {}; 
\node [below=0.5cm of d.south east] (n5) {}; 
\node [below=0.5cm of e.south east] (fine) {}; 
\draw[|-|,dashed] (inizio.north) -- node[label=below:$0$]{} (n1.north);
\draw[|-|,dashed] (n1.north)-- node[label=below:$1$]{} (n2.north); 
\draw[|-|,dashed] (n2.north)-- node[label=below:$2$]{}(n3.north);
\draw[|-|,dashed] (n3.north)-- node[label=below:$3$]{}(n4.north);
\draw[|-|,dashed] (n4.north)-- node[label=below:$4$]{}(n5.north); 
\draw[|-|,dashed] (n5.north)-- node[label=below:$5$]{}(fine.north);

\end{tikzpicture}
\caption{Non-admissible tiling (case 2).}
\label{fig:n7}
\end{subfigure}%
\hfill
\begin{subfigure}{.5\textwidth}
  \begin{tikzpicture}
\node[quadrato] (a) {} ;
\node[quadrato, above= 0cm of a] (a2) {} ;
\node[domino, fill=red, right =0cm of a] (b) {};
\node[domino,right=0cm of b] (c) {};
\node[domino,above=0cm of c] (c2) {};
\node[quadrato,right=0cm of c] (d) {};

\node [below=0.5cm of a.south west] (inizio) {}; 
\node [right=1cm of inizio.west] (n1) {}; 
\node [right=1cm of n1.west] (n2) {}; 
\node [right=1cm of n2.west] (n3) {}; 
\node [right=1cm of n3.west] (n4) {}; 
\node [right=1 cm of n4.west] (n5) {}; 
\node [below=0.5cm of d.south east] (fine) {}; 
\draw[|-|,dashed] (inizio.north) -- node[label=below:$0$]{} (n1.north);
\draw[|-|,dashed] (n1.north)-- node[label=below:$1$]{} (n2.north); 
\draw[|-|,dashed] (n2.north)-- node[label=below:$2$]{}(n3.north);
\draw[|-|,dashed] (n3.north)-- node[label=below:$3$]{}(n4.north);
\draw[|-|,dashed] (n4.north)-- node[label=below:$4$]{}(n5.north); 
\draw[|-|,dashed] (n5.north)-- node[label=below:$5$]{}(fine.north); 

\end{tikzpicture}
\caption{Non-admissible tiling (case 3b).}
\label{fig:n8}
\end{subfigure}%
\end{figure}

\begin{figure}[!hbt]
\addtocounter{figure}{-1} 
    \centering
\begin{subfigure}{.5\textwidth}
 \addtocounter{subfigure}{2}
   \begin{tikzpicture}
\node[quadrato] (a) {} ;
\node[quadrato, above= 0cm of a] (a2) {} ;
\node[domino, right =0cm of a] (b) {} ;
\node[domino, above =0cm of b] (b2) {} ;
\node[domino, above =0cm of b2] (b3) {} ;
\node[quadrato, right =0cm of b] (c) {} ;
\node[quadrato, above =0cm of c] (c2) {} ;
\node[quadrato,right=0cm of c] (d) {};
\node[quadrato,above=0cm of d] (d2) {};
\node[quadrato, fill=red, right =0cm of d] (e) {} ;
\node[quadrato, fill=red, above =0cm of e] (e2) {} ;

\node [below=0.5cm of a.south west] (inizio) {}; 
\node [right=1cm of inizio.west] (n1) {}; 
\node [right=1cm of n1.west] (n2) {}; 
\node [below=0.5cm of b.south east] (n3) {}; 
\node [right=1cm of n3.west] (n4) {}; 
\node [below=0.5cm of d.south east] (n5) {}; 
\node [below=0.5cm of e.south east] (fine) {}; 
\draw[|-|,dashed] (inizio.north) -- node[label=below:$0$]{} (n1.north);
\draw[|-|,dashed] (n1.north)-- node[label=below:$1$]{} (n2.north); 
\draw[|-|,dashed] (n2.north)-- node[label=below:$2$]{}(n3.north);
\draw[|-|,dashed] (n3.north)-- node[label=below:$3$]{}(n4.north);
\draw[|-|,dashed] (n4.north)-- node[label=below:$4$]{}(n5.north); 
\draw[|-|,dashed] (n5.north)-- node[label=below:$5$]{}(fine.north);

\end{tikzpicture}
\caption{Non-admissible tiling (case 4a).}
\label{fig:n9}
\end{subfigure}%
\hfill
\begin{subfigure}{.5\textwidth}
    \begin{tikzpicture}
\node[quadrato] (a) {} ;
\node[quadrato, above= 0cm of a] (a2) {} ;
\node[domino, right =0cm of a] (b) {} ;
\node[domino, above =0cm of b] (b2) {} ;
\node[domino, above =0cm of b2] (b3) {} ;
\node[quadrato, right =0cm of b] (c) {} ;
\node[quadrato,right=0cm of c] (d) {};
\node[quadrato,above=0cm of d] (d2) {};
\node[quadrato, fill=red, right =0cm of d] (e) {} ;

\node [below=0.5cm of a.south west] (inizio) {}; 
\node [right=1cm of inizio.west] (n1) {}; 
\node [right=1cm of n1.west] (n2) {}; 
\node [below=0.5cm of b.south east] (n3) {}; 
\node [right=1cm of n3.west] (n4) {}; 
\node [below=0.5cm of d.south east] (n5) {}; 
\node [below=0.5cm of e.south east] (fine) {}; 
\draw[|-|,dashed] (inizio.north) -- node[label=below:$0$]{} (n1.north);
\draw[|-|,dashed] (n1.north)-- node[label=below:$1$]{} (n2.north); 
\draw[|-|,dashed] (n2.north)-- node[label=below:$2$]{}(n3.north);
\draw[|-|,dashed] (n3.north)-- node[label=below:$3$]{}(n4.north);
\draw[|-|,dashed] (n4.north)-- node[label=below:$4$]{}(n5.north); 
\draw[|-|,dashed] (n5.north)-- node[label=below:$5$]{}(fine.north);

\end{tikzpicture}
\caption{Non-admissible tiling (case 4b).}
\label{fig:n10}
\end{subfigure}%
 \caption{Some examples of non-admissible tilings.}
 \label{fig:n11}
\end{figure}
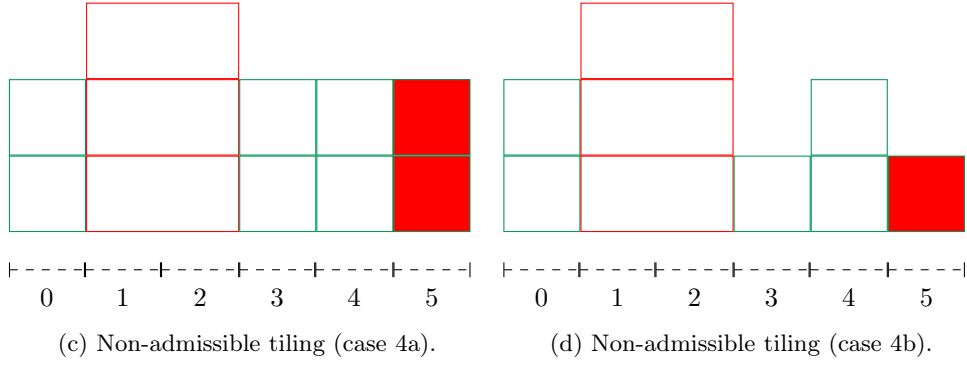
\FloatBarrier